\documentclass{amsart}
\usepackage{latexsym,amssymb,amsfonts,amsmath, graphicx, color,enumerate,mathbbol}
\usepackage[latin1]{inputenc}
\usepackage[pdftex]{hyperref}

\numberwithin{equation}{section}
\newtheorem{theo}{Theorem}

\newtheorem{lemma}[theo]{Lemma}
\newtheorem{prop}[theo]{Proposition}
\newtheorem{cor}[theo]{Corollary}

\newtheorem{assum}[theo]{Assumption}

\theoremstyle{definition}

\newtheorem{rem}[theo]{Remark}

\newtheorem{exa}[theo]{Example}

\numberwithin{theo}{section}
\def\mv{\mathsf{v}}

\def\:{\thinspace:\thinspace}
\def\linie{\vrule height 14pt depth 5pt}
\def\back{\noalign{\vskip-3pt}}

\def\@secappcntformat#1{%
 \ifappendix \rm\appendixname\ifoneappendix\else~\fi\fi 
 \ifoneappendix \else \csname the#1\endcsname\relax\fi
 \ifx\apphe@d\@empty \else .\fi\enskip   
}

\author{Delio Mugnolo}

\title{Damped wave equations with dynamic boundary conditions}

\address{Institut f\"ur Angewandte Analysis, Universit\"at Ulm, Helmholtzstra{\ss}e 18, D-89081 Ulm, Germany}
\email{delio.mugnolo@uni-ulm.de}

\begin{document}

\begin{abstract}
We discuss several classes of linear second order initial-boundary value problems, where damping terms appear in the main wave equation as well as in the {dynamic} boundary condition. We investigate their well-posedness and describe some qualitative properties of their solutions, including boundedness, stability, or almost periodicity. In particular, we are able to characterize the analyticity of certain $C_0$-semigroups associated to such problems. Applications to several problems on domains and networks are shown, mostly borrowed from~\cite{CENP05,XL04}.
\end{abstract}

\maketitle

\section{Introduction}

In recent years, wave and beam equations with dynamic boundary conditions have been studied by many authors, see e.g. \cite{LSY93, GV94, AKS96, FGGR01, GGG03, XL04, XL04b, BE05, CENP05, KW06, Mu05} and references therein.  Wave equations with such boundary conditions are motivated by physical models incorporating the effect of frictions, as shown in~\cite{Go06}. Oscillating models involving dynamic boundary conditions for networks or more general polygonal domains have been considered, among others, by Ali Mehmeti, cf.~\cite[Chapt.~4]{Al94} and references therein, as well as by Lagnese--Leugering--Schmidt in~\cite[\S~2.7]{LLS94}. In fact, our setting can be adapted to problems on networks, interface problems and domains alike, cf. examples below.

Most recent papers deal with abstract methods based on the theories of operator matrices and $C_0$-semigroups. These theories may have advantages over more usual methods based on Hilbert space methods and energy estimates. They allow more flexibility in treating non-dissipative systems. While most of the above mentioned papers only treat undamped wave equations, aim of this paper is to apply known methods in order to investigate a class of \emph{damped} problems. More precisely, we consider second order problems where the damping effect can be observed in both the waveguide and its boundary.

While the first-order counterpart of this setting, i.e., \emph{diffusion problems} with dynamic boundary conditions, has been often discussed both on domains and on ramified structures due to its relations to stochastic analysis (in particular to the theory of Feller semigroups), comparatively less attention has been devoted to \emph{wave equations} with dynamic boundary conditions. Though, the connections between wave equations with further kinds of oscillatory boundary conditions (which are well-known in the mathematical physics of acoustic waves and mixed water/ice systems, cf.~\cite{Kr61,Be76,Bl00}) and dynamic ones has been thoroughly shown in~\cite{GGG03}. Wave equations with different kinds of time-dependent boundary conditions have been recently considered by Nicaise and coauthors, cf.~\cite{NP07} and references therein.  A numeric approach to this class of problems based on a Trotter--Kato-type result has been proposed in~\cite{LNX08}. See also the conference procee!
 dings~\cite{ABN01}.

We stress that our theory is formulated in the abstract context of Banach spaces, whereas waves equations are usually discussed in a Hilbert space framework. This is due to the fact that, by a celebrated result of Littman, cf.~\cite{Li63}, \emph{undamped} wave equations are well-posed in an $L^p$-setting if and only if $p=2$ or the space dimension is 1. However, more recent results indicate that this limitation does not apply to the \emph{damped} case, see e.g.~\cite{CS05,CS08}.

\section{Mathematical framework}

\begin{assum}\label{assodampbasic} We impose the following throughout this paper.
\begin{enumerate}[(1)]
\item $X$, $Y$, and $\partial X$ are Banach spaces
such that $Y\hookrightarrow X$.
\item ${A}:D({A})\subset X\to X$ and $C:D({C})\subset X\to X$ are linear operators.
\item $L:D(A)\cap D(C)\to {\partial X}$ is a linear and surjective operator.
\item $B_1:D(A)\to\partial X$ and $B_2:D(C)\to\partial X$ are linear operators. 
\item $B_3:D(B_3)\subset {\partial X}\to{\partial X}$ and $B_4:D(B_4)\subset {\partial X}\to{\partial X}$ are linear and closed operators.
\end{enumerate}
\end{assum}

Functions on the main waveguide will be throughout this paper vectors in a Banach space $X$. We introduce a complete abstract second order problem 
$$\ddot{u}(t)= Au(t)+C\dot{u}(t),\qquad t{\geq 0}$$
and equip such a problem with second order dynamic boundary conditions represented by an equation
\begin{equation}\label{dbc}
\ddot{w}(t)= B_1u(t)+B_2\dot{u}(t)+B_3w(t)+B_4\dot{w}(t), \qquad t{\geq 0},
\end{equation}
on another Banach space $\partial X$. Here the relation between the variables $u$ and $w$ is expressed by
$$w(t)=Lu(t) \qquad\hbox{and/or by}\qquad \dot{w}(t)=L\dot{u}(t),\qquad t\geq 0,$$
where $L$ is some operator from $X$ to $\partial X$.
We want to investigate (analytic) well-posedness and asymptoptic behavior of such a system. To this aim we re-write it in an abstract form and are eventually led to the complete second order abstract Cauchy problem
\begin{equation}\tag{c${\mathcal{ACP}}^2$}
\left\{
\begin{array}{rcl}
 \ddot{\mathfrak u}(t)&=&\mathcal{A}{\mathfrak u}(t)+\mathcal{C}\dot{\mathfrak u}(t),\qquad t\geq0,\\
 {\mathfrak u}(0)&=&{\mathfrak f}\in{\mathcal X},\qquad \dot{\mathfrak u}(0)={\mathfrak g} \in{\mathcal X},
\end{array}
\right.
\end{equation}
on the product space 
\begin{equation*}
{\mathcal X}:=X\times \partial X.
\end{equation*}
Here 
\begin{equation}\label{accorsivo}
\mathcal{A}:=\begin{pmatrix}
A & 0\\
B_1 & B_3
\end{pmatrix}
\qquad\hbox{and}\qquad
\mathcal{C}:=\begin{pmatrix}
C & 0\\
B_2 & B_4
\end{pmatrix}
\end{equation}
are operator matrices on $\mathcal X$, and their domains will depend on \emph{how unbounded the damping term $C$ is with respect to the elastic term $A$}, as we see next. Moreover, the new variable $\mathfrak{u}$ is to be understood as
$${\mathfrak u}(t):=\begin{pmatrix}u(t)\\ Lu(t)\end{pmatrix},\qquad t\geq 0.$$
If we reduce the second order evolution problem in $({\rm c}\mathcal{ACP}^2)$ to a first order abstract Cauchy problem, our goal becomes to discuss the  well-posedness of
\begin{equation}\tag{${\mathbb{ACP}}$}
\left\{
\begin{array}{rcl}
 \dot{\mathbb u}(t)&=&{\mathbb A}{\mathbb u}(t),\qquad t\geq0,\\
 {\mathbb u}(0)&=&{\mathbb f},
\end{array}
\right.
\end{equation}
in some suitable phase space, where $\mathbb A$ is the reduction matrix defined by
\begin{equation}\label{dampabb}
{\mathbb A}:=
\begin{pmatrix}
0 & I_{D({\mathcal C})}\\
{\mathcal A} & {\mathcal C}
\end{pmatrix}
\end{equation}
with suitable domain, and accordingly
$${\mathbb u}(t)\equiv\begin{pmatrix}{\mathfrak u}(t)\\\dot{\mathfrak u}(t)\end{pmatrix},\quad t\geq 0,\qquad
 {\mathbb f}\equiv\begin{pmatrix}{\mathfrak f}\\ {\mathfrak g}\end{pmatrix}.$$

A setting of this kind permits to treat damped wave equations on domains or, more generally, on networks or ramified structures.

\begin{exa}\label{motiv}
Consider the initial value problems associated with a one dimensional damped plate-like equation on a bounded interval
\begin{equation}\label{motiv1}
\left\{
\begin{array}{rcll}
\ddot{u}(t,x)&=&-u''''(t,x)+\dot{u}''(t,x), &t\geq 0,\; x\in(0,1),\\
u''(t,j)&=&(-1)^{j} u'(t,j)-u(t,j), &t\geq 0,\; j=0,1,\\
\ddot{u}(t,j)&=&(-1)^{j+1} u'''(t,j)+(-1)^{j} u'(t,j)\\
&&\qquad +(-1)^j \dot{u}'(t,j)-u(t,j)-\dot{u}(t,j), &t\geq 0,\; j=0,1,
\end{array}
\right.
\end{equation}
or with a strongly damped wave equation on an \emph{open book} consisting of $N$ copies of a domain $\Omega\subset{\mathbb R}^n$
\begin{equation}\label{motiv2}
\left\{
\begin{array}{rcll}
 \ddot{u}_j(t,x)&=& \Delta (\alpha u_j+\dot{u}_j)(t,x), &t\geq 0,\; x\in \Omega,\; j=1,\ldots,N,\\
u_j(t,z)&=&u_\ell(t,z)=:u(t,z), &t\geq 0,\; z\in\partial\Omega,\; j,\ell=1,\ldots,N,\\
 \ddot{u}_j(t,z)&=& -\frac{\partial}{\partial\nu} (\beta u_j+\dot{u}_j)(t,z)+\gamma u(t,z)+\delta \dot{u}(t,z), &t\geq 0,\; z\in\partial\Omega,\; j=1,\ldots,N,
\end{array}
\right.
\end{equation}
for $\alpha,\beta,\gamma,\delta\in\mathbb C$. Here the second equation prescribes continuity along the \emph{binding} of the book.

The initial value problems associated to both systems can be reduced to $(c{\mathcal{ACP}}^2)$ on the Hilbert spaces $L^2(0,1)\times{\mathbb C}^2$ and $L^2(\Omega)\times L^2(\partial \Omega)$, respectively. 
Certain assumptions relating boundedness of elastic and damping term are satisfied, and by known results on damped wave equations  we deduce analytic well-posedness.\qed
\end{exa}

In most usual examples (like in~\eqref{motiv1} and~\eqref{motiv2}), known energy estimates (cf.~\cite{FGGR01,AMPR03}) permit to apply known results on damped second order problems -- c.f.~\cite{Ne86,CT88,CT89,XL98,Mu08}, \cite[\S~XVIII.5.1]{DL88}, or~\cite[\S~VI.3]{EN00}, for methods based on spectral theory, functional calculus, operator matrix theory, of forms.


All these methods require certain boundedness assumptions relating $A,C$ and $B_3,B_4$. 
Aim of this paper is to give sufficient conditions on $A,C,L,B_1,B_2,B_3,B_4$ that ensure that the reduction matrix $\mathbb A$ generates a $C_0$-semigroup, independently of dissipativeness. The theorems we are going to present in the remainder of this paper are of the following form: 

\emph{If the damped wave equations with homogeneous boundary conditions associated with (suitable restrictions of) the matrix operators $\mathcal A$ and $\mathcal C$ are well-posed, then so is $(c\mathcal{ACP}^2)$.}

We will be specially concerned with investigatig further qualitative properties (boundedness, compactness, almost periodicity...) enjoyed by such a semigroup. 
In the following we will pay special attention to second order problems that are governed by analytic semigroups, a feature often discussed in applications, cf.~\cite[\S~1]{CR82}. 
Furthermore, a parabolic character directly implies regularity results, and can thus be exploited in order to discuss  semilinear problems,~e.g. by the techniques presented in~\cite{Lu95}. 
Generation of analytic semigroups in the context of damped wave equations with dynamic boundary conditions has also been investigated in~\cite{XL04,XL04b} by different methods.

Let us explain the plan of this paper. We have introduced in the Assumption~\ref{assodampbasic}.(1) a Banach space $Y$. Such a space $Y$ is in common applications somehow related to the domain of the operator $C$ and to the phase space of the second order problem -- it was fact $Y=H^2(0,1)\cap H^1_0(0,1)$ and $Y=H^1_0(\Omega)$ in~\eqref{motiv1} and~\eqref{motiv2}, respectively. (In fact, in concrete cases it will be a Sobolev space of the same order of the so-called \emph{Kisy\'nski space} of the wave equation, i.e., the first factor of the phase space.) Depending on $Y$ and on the operator $L$, we need as in~\cite{Mu05} to distinguish three different cases: $L$ can be 

-- unbounded from $Y$ to $\partial X$,

-- unbounded from $X$ to $\partial X$ but bounded from $Y$ to $\partial X$, or

-- bounded from $X$ to $\partial X$. 

In this paper \emph{we only consider the first two cases}, in Sections~\ref{sect3} and~\ref{sect4}, respectively. These occur, e.g., when we consider a wave equation on an $L^p$-space and $L$ is the normal derivative (see Example~\ref{exas2}) or the trace operator (see Example~\ref{exas3}), respectively. Our results should be compared with those of~\cite{CENP05,XL04, XL04b}. Instead, the case of $L\in{\mathcal L}(X,\partial X)$ is typical for spaces $X$ where the point evaluation is a bounded operator. 
The strongly damped case, i.e., the case of an operator $C$ that is ``more unbounded" than $A$, is technically slightly different and will be treated in Sections~\ref{sect5}.

Finally, in Section~\ref{appe} we prove a technical lemma on the exponential stability of semigroups generated by operator matrices. This seems to be new and of independent interest.

\section{The damped case $L\not\in{\mathcal L}(Y,\partial X)$}\label{sect3}

Of concern in this section are second order abstract problems with dynamic boundary conditions of the form
\begin{equation}\tag{AIBPV$_a^2$}
\left\{
\begin{array}{rcll}
 \ddot{u}(t)&=& Au(t)+C\dot{u}(t), &t{\geq 0},\\
 \ddot{w}(t)&=& B_1u(t)+ B_2\dot{u}(t)+B_3 w(t)+B_4\dot{w}(t), &t{\geq 0},\\
 w(t)&=&Lu(t), &t{\geq 0},\\
 u(0)&=&f\in X, \qquad\;\;\dot{u}(0)=g\in X,&\\
 x(0)&=&h\in \partial X, \qquad\dot{x}(0)=j\in \partial X.&
\end{array}
\right.
\end{equation}


 \begin{assum}\label{assdamp}
We complement the Assumptions~\ref{assodampbasic} by the following.
\begin{enumerate}[(1)]
\item $\begin{pmatrix} A\\ L\end{pmatrix}:D(A)\subset X\to X\times \partial X$ is closed.
\item $A_0:=A_{|{\rm ker}(L)}$ has nonempty resolvent set.
\item $C$ is closed, $D(A)\subset D(C)$, and $[D(C)]$ is isomorphic to $Y$.
\item $\partial Y$ is a Banach space, $[D(B_4)]$ is isomorphic to $\partial Y$ and $\partial Y\hookrightarrow \partial X$.
\end{enumerate}
\end{assum}

We denote by $[D(A)_L]$ the Banach space obtained by endowing $D(A)$ with the graph norm of the closed operator $A\choose L$. 
Observe that by the Closed Graph Theorem the embeddings $[D(A_0)]\hookrightarrow [D(A)_L]\hookrightarrow Y$ hold.
Furthermore, for $\lambda\in\rho(A_0)$ we consider the \emph{Dirichlet operators} $D^{A,L}_\lambda$ associated with $A,L$, cf.~\cite[Lemma~2.3]{CENN03}. Such operators are right inverses of $L$ for all $\lambda\in\rho(A_0)$ and are such that $AD_\lambda^{A,L}=\lambda D_\lambda^{A,L}$. Moreover, they are linear and bounded from $\partial X$ to $[D(A)_L]$, cf.~\cite[Lemma~3.2]{Mu05}.

The following can be verified by a direct matrix computation.

\begin{lemma}
Consider the operator 
\begin{equation}\label{dampabb2}
{\mathbb A}:=
\begin{pmatrix}
0 & 0 & I_Y & 0\\
0 & 0 & 0 & I_{\partial Y}\\
A & 0 & C & 0\\
B_1 & B_3 & B_2 & B_4
\end{pmatrix}.
\end{equation}
with domain
\begin{equation}\label{dampabb2dom}
D({\mathbb A}):=\left\{
\begin{pmatrix}u\\ x\\ v\\ y\end{pmatrix}
\in D(A)\times D(B_3)\times D(C)\times D(B_4): Lu=x
\right\}
\end{equation}
on the Banach space 
\begin{equation}\label{xxx}
{\mathbb X}:=Y\times \partial Y\times X\times \partial X.
\end{equation}
Then the well-posedness of the first order abstract Cauchy problem $(\mathbb{ACP})$ on $\mathbb X$ is equivalent to the well-posedness of  $({\rm AIBPV}_a^2)$ on $X$ and $ \partial X$.
\end{lemma}

We can identify a general function ${\mathbb u}:{\mathbb R}_+\to {\mathbb X}$ by
$$\mathbb{u}(t)\equiv\begin{pmatrix}u(t)\\ x(t)\\ v(t)\\ y(t)\end{pmatrix},\qquad t\geq 0.$$
Hence, if ${\mathbb u}$ is a classical solution to $(\mathbb{ACP})$, so that ${\mathbb u}\in C^1({\mathbb R}_+,{\mathbb X})\cap C({\mathbb R}_+, [D({\mathbb A})])$, then moreover $\dot{u}(t)=v(t)$, $t\geq 0$, and $v\in C^1({\mathbb R}_+,X)$, and we conclude that $u\in C^2({\mathbb R}_+,X)\cap C^1({\mathbb R}_+,Y)$. 

\begin{lemma}\label{simildamp}
Let $\lambda\in\rho(A_0)$. The operator matrix $({\mathbb A},D({\mathbb A}))$ on $\mathbb X$ is similar to 
\begin{equation}\label{formg}
{\mathbb G}:=
\begin{pmatrix}
\lineskip=0pt
0  & I_Y & \linie & 0  & -D^{A,L}_\lambda\\ \back
A_0  & C &\linie  & 	\lambda D^{A,L}_\lambda  & 0\\ 
\noalign{\hrule}
0  & 0  & \linie & 0  & I_{\partial Y}\\ \back
B_1 & B_2 & \linie & B_3 + B_1 D^{A,L}_\lambda & B_4\\
\end{pmatrix}
\end{equation}
with domain
\begin{equation*}
D({\mathbb G}):= D(A_0)\times D(C)\times D(B_3)\times Y)
\end{equation*}
on the Banach space 
$${\mathbb Y}:=Y\times X\times \partial Y\times \partial X.$$
\end{lemma}


\begin{proof}
Let $\lambda\in\rho(A_0)$. First of all, we consider an isomorphism of $\mathbb X$ onto $\mathbb Y$ given by
\begin{equation*}
{\mathbb U}_\lambda:=
\begin{pmatrix}
I_Y & -D^{A,L}_\lambda & 0 & 0\\
0 & 0 & I_X & 0\\
0 & I_{\partial Y} & 0 & 0\\
0 & 0 & 0 & I_{\partial X}\\
\end{pmatrix},
\quad\hbox{with inverse}\quad
{\mathbb U}_\lambda^{-1}:=
\begin{pmatrix}
I_Y &  & D_\lambda^{A,L} & 0\\
0 & 0 & I_{\partial Y} & 0\\
0 & I_X & 0 & 0\\
0 & 0 & 0 & I_{\partial X}\\
\end{pmatrix}.
\end{equation*}
We want to compute ${\mathbb G}={\mathbb U}_\lambda {\mathbb A} {\mathbb U}^{-1}_\lambda$. Observe that 
\begin{eqnarray*}
D({\mathbb G})&=&\{{\mathbb x}\in{\mathbb X}: {\mathbb U}^{-1}_\lambda {\mathbb x}\in D({\mathbb A})\}\\
&=&\left\{\begin{pmatrix}u\\ v\\ x\\ y\end{pmatrix}\in {\mathbb X}: \begin{pmatrix}u+D^{A,L}_\lambda x\\ x\\ v\\ y\end{pmatrix}\!\!\!
\begin{array}{ll}
&\in D(A)\times D(B_3)\times D(C)\times  D(B_4)\\
&\qquad\hbox{and } L(u+D^{A,L}_\lambda x) =x
\end{array}
\right\}\\
&=&\left\{\begin{pmatrix}u\\ v\\ x\\ y\end{pmatrix}\in {\mathbb X}: \begin{pmatrix}u+D^{A,L}_\lambda x\\ x\\ v\\ y\end{pmatrix}\!\!\!
\begin{array}{ll}
&\in D(A)\times D(B_3)\times D(C)\times  D(B_4)\\
&\qquad\hbox{and } Lu=0
\end{array}
\right\}\\
&=& D(A_0)\times D(C)\times D(B_3)\times D(B_4).
\end{eqnarray*}
Moreover,
\begin{eqnarray*}
{\mathbb U}_\lambda {\mathbb A} {\mathbb U}^{-1}_\lambda\begin{pmatrix}u\\ v\\ x\\ y\end{pmatrix} &=&{\mathbb U}_\lambda {\mathbb A} \begin{pmatrix}u+D^{A,L}_\lambda x\\ x\\ v\\ y\end{pmatrix}\\
&=& {\mathbb U}_\lambda \begin{pmatrix}v\\ y\\ A_0 u + \lambda D^{A,L}_\lambda x+Cv\\ B_1 (u+D^{A,L}_\lambda x) + B_3 x + B_2 v + B_4 y\end{pmatrix}\\
&=& \begin{pmatrix}v-D^{A,L}_\lambda y\\A_0 u + \lambda D^{A,L}_\lambda x+Cv\\  y\\ B_1 (u+D^{A,L}_\lambda x) + B_3 x + B_2 v + B_4 y\end{pmatrix}.
\end{eqnarray*}
This finally shows the claimed representation of the operator matrix ${\mathbb G}$.
\end{proof}

We are now in the position to prove the main result of this section.

\begin{theo}\label{maindampun}
Under the Assumptions~\ref{assodampbasic} and~\ref{assdamp}, let $B_1\in{\mathcal L}([D(A)_L],\partial X)$. Then  the following assertions hold.
\begin{enumerate}[(1)]
\item Assume that $B_1\in{\mathcal L}([D(A_0)],\partial Y)$, or else that $B_1\in{\mathcal L}(Y,\partial X)$, and moreover that $B_2\in{\mathcal L}(Y,\partial Y)$, or else that $B_2\in{\mathcal L}(X,\partial X)$. 
Then the operator matrix $\mathbb A$ generates a $C_0$-semigroup on $\mathbb X$ if and only if
\begin{equation}\label{redmata1}
\begin{pmatrix}
0 & I_Y\\
A_0 & C
\end{pmatrix}
\quad\hbox{with domain}\quad D({A}_0)\times Y
\end{equation}
and
\begin{equation}\label{redmata2}
\begin{pmatrix}
0 & I_{\partial Y}\\
B_3 & B_4
\end{pmatrix}
\quad\hbox{with domain}\quad D(B_3)\times D(B_4)
\end{equation}
generate $C_0$-semigroups on $Y\times X$ and $\partial Y\times \partial X$, respectively.

\item Let $B_2\in{\mathcal L}(Y,\partial X)$. Then the reduction matrices introduced in~\eqref{redmata1}--\eqref{redmata2} both generate analytic semigroups if and only if $\mathbb A$ generates an analytic semigroup.

\item Let $B_1=B_2=0$. If the semigroup generated by either of the  matrices defined in~\eqref{redmata1}--\eqref{redmata2} is bounded and the other one is uniformly exponentially  stable, then the semigroup generated by ${\mathbb A}$ is bounded. 

\item Let $B_1=B_2=0$. Assume the semigroups generated by matrices in~\eqref{redmata1}--\eqref{redmata2} to be bounded. Let further the semigroup generated by the matrix in~\eqref{redmata1} be analytic. If the matrices  in~\eqref{redmata1}--\eqref{redmata2} have no common purely imaginary spectral values, then the semigroup generated by ${\mathbb A}$ is bounded.

\item Let $B_1\in{\mathcal L}(Y,\partial X)$ and $B_2\in{\mathcal L}(X,\partial X)$. If both semigroups generated by matrices in~\eqref{redmata1}--\eqref{redmata2} are uniformly exponentially stable, then there exists $\epsilon>0$ such that also the semigroup generated by ${\mathbb A}$ is uniformly exponentially stable, provided that $\Vert B_1\Vert_{{\mathcal L}(Y,\partial X)} + \Vert B_2\Vert_{{\mathcal L}(X,\partial X)}<\epsilon$.

\item The operator matrix $\mathbb A$ on $\mathbb X$ has compact resolvent if and only if all the embeddings $[D({A}_0)]\hookrightarrow Y\hookrightarrow X$ and $[D(B_3)]\hookrightarrow \partial Y\hookrightarrow \partial X$ are compact. 
\end{enumerate}
\end{theo}

\begin{proof}
Let $\lambda\in\rho(A_0)$. By Lemma~\ref{simildamp} instead of $\mathbb A$ on $\mathbb X$ it suffices to investigate the similar operator matrix $\mathbb G$ on $\mathbb Y$. We consider $\mathbb G$ as a $2\times 2$ operator matrix with diagonal domain. More precisely,
$${\mathbb G}={\mathbb G}_0+{\mathbb G}_1:=
\begin{pmatrix}
{\bf A} & 0\\
{\bf C} & {{\bf D}}
\end{pmatrix}+
\begin{pmatrix}
0 & {\bf B}\\
0 & 0
\end{pmatrix}, 
\qquad D(\mathbb{G})=D({\bf A})\times D({{\bf D}}),$$
where the $2\times 2$ block-entries ${\bf A},{\bf B},{\bf C},{\bf D}$ are defined as in~\eqref{formg}. Also observe that by assumption the operator $B_1 D^{A,L}_\lambda$ is bounded from $\partial Y$ to $\partial X$ for all $\lambda\in\rho(A_0)$, so that we can discuss the generator property of the reduction matrix introduced in~\eqref{redmata2} instead of $\bf D$. By the Assumptions~\ref{assdamp} the Dirichlet operator $D^{A,L}_\lambda$ is bounded from $\partial X$ to $Y$, hence the block entry $\bf B$ is bounded from $\partial Y\times \partial X$ to $Y\times X$. Thus, by the bounded perturbation theorem we only have to care about the lower triangular operator matrix ${\mathbb G}_0$.

(1) The off-diagonal block-entry $\bf C$ is bounded from $[D({\bf A})]$ to $[D({{\bf D}})]$ or from $Y\times X$ to $\partial Y\times \partial X$. It follows by a perturbation result by Desch--Schappacher or by the bounded perturbation theorem, respectively, that ${\mathbb G}_0$ generates a $C_0$-semigroup on $\mathbb Y$ if and only if both diagonal block-entries ${\bf A},{\bf D}$ of ${\mathbb G}_0$ generate $C_0$-semigroups on $Y\times X$ and on $\partial Y\times \partial X$, respectively.

(2) The diagonal block-entries of ${\mathbb G}_0$ both generate analytic semigroups. Moreover, the off-diagonal entry $\bf B$ is bounded from $[D({\bf A})]$ to $\partial Y\times \partial X$. It follows by~\cite[Cor.~3.3]{Na89} that ${\mathbb G}_0$ generates an analytic semigroup on $\mathbb Y$.

(3)--(4)--(5) These assertions follow directly from Proposition~\ref{complete} below.

(6) Since $D({\mathbb G})=D({\bf A})\times D({\bf D})$,  $\mathbb G$ has compact resolvent if and only its diagonal block-entries have compact resolvent.
\end{proof}

\begin{exa}\label{exas2}
We discuss the initial value problem associated with
\begin{equation*}
\left\{
\begin{array}{rcll}
 \ddot{u}(t,x)&=& -\Delta^2 u(t,x)+\Delta\dot{u}(t,x), &t{\geq 0},\; x\in\Omega\\
 \ddot{w}(t,z)&=& q_1(z)\Delta u(t,z)+ q_2(z)\dot{u}(t,z)\\
&&\qquad - \Delta_{\partial\Omega}^2 w(t,z)- w(t,z)+ \Delta_{\partial\Omega} \dot{w}(t,z)-\dot{w}(t,z), &t{\geq 0},\; z\in\partial\Omega,\\
 w(t,z)&=&\frac{\partial \Delta u}{\partial \nu}(t,z) - p(z)\Delta u(t,z), &t{\geq 0},\; z\in\partial \Omega,\\
\frac{\partial {u}}{\partial\nu}(t,z)&=&p(z){u}(t,z), &t{\geq 0},\; z\in\partial \Omega,\\
\frac{\partial \dot{u}}{\partial\nu}(t,z)&=&p(z)\dot{u}(t,z), &t{\geq 0},\; z\in\partial \Omega,
\end{array}
\right.
\end{equation*}
similar to that discussed in~\cite[Exa.~4.3]{XL04b}. Here $\Omega\subset{\mathbb R}^n$ is a bounded open domain with smooth boundary $\partial\Omega$ and $p,q_1,q_2\in L^\infty(\partial\Omega)$, $p<0$. Observe that -- whenever re-written as $({\rm c} \mathcal{ACP}^2)$ -- the operator matrix $\mathcal C$ is in general neither self-adjoint, nor strictly negative definite, and $\mathcal{A}\not=-{\mathcal C}^2$, thus it is not possible to directly apply the results presented in~\cite{CT89}, \cite[\S~XVIII.5.1]{DL88}, or~\cite[\S~6.4]{XL98}.

In order to apply the results presented above, we consider
$$Y:=\left\{u\in H^2(\Omega): \frac{\partial u}{\partial\nu}=pu_{|\partial\Omega}\right\},\quad X:=L^2(\Omega),\quad \partial Y:=H^2(\partial \Omega),\quad \partial X:=L^2(\partial \Omega),$$
and further
$$C:=\Delta,\qquad D(C):=Y,$$
i.e., $C$ is the Laplacian with Robin boundary conditions, and
$$A:=-\Delta^2,\qquad D(A):=\left\{u\in H^\frac{7}{2}(\Omega): \frac{\partial u}{\partial\nu} = p u_{|\partial\Omega} \right\}\subset D(C).$$
Let 
$$Lu(z):=\frac{\partial \Delta u}{\partial \nu}(z) - p(z)\Delta u(z)\qquad\hbox{for}\; u\in H^\frac{7}{2}(\Omega),\; z\in\partial\Omega.$$
Such an operator is well defined in the sense of traces. Then by usual boundary regularity results one sees that $-A_0=-A_{|\ker(L)}$ is the square of $C$. The operator $C$ is self-adjoint and strictly negative definite, and we obtain by~\cite[Thm.~6.4.3 and Thm.~6.4.4]{XL98} that the operator matrix defined in~\eqref{redmata1}	
generates an analytic, uniformly exponentially stable semigroup on $Y\times X$. Let now
$$(B_1 u)(z)  := q_1(z)\Delta u(z)\quad (B_2u)(z):= q_2(z)u(z),\qquad u\in H^2(\Omega),\; z\in\partial \Omega.$$
It is clear that $B_1,B_2$ are bounded from $Y$ to $\partial X$ and from $X$ to $\partial X$, respectively. Consider moreover the operators $B_3$ and $B_4$ defined by	
\begin{equation*}
\begin{array}{lll}
&B_3:= -\Delta^2_{\partial\Omega}-I,\qquad &D(B_3):=H^4(\partial\Omega),\\
&B_4:=\Delta_{\partial\Omega}-I, &D(B_4):= H^2(\partial\Omega),
\end{array}
\end{equation*}
where $\Delta_{\partial\Omega}$ denotes the Laplace--Beltrami operator, which is self-adjoint and negative definite. By~\cite[Thm.~1.1]{CT89} the operator matrix defined in~\eqref{redmata2} generates a uniformly exponentially stable analytic semigroup on $\partial Y\times \partial X$.
%
If $\| q_1\|_\infty+\Vert q_2\Vert_\infty\to 0$, then  by~Theorem~\ref{maindampun}.(5) the solution $u$ converges to 0 in the energy norm. By Theorem~\ref{maindampun}.(4), the semigroup governing the problem is also compact.
%
\qed\end{exa}

\section{The damped case $L\in{\mathcal L}(Y,\partial X)$}\label{sect4}

The case of $L\in{\mathcal L}(Y,\partial X)$ introduces some technical difficulties. In particular, we will show that our damped wave equations is well-posed on a phase space that is \emph{not} a product space (as $\mathbb X:=Y\times \partial Y\times X\times \partial X$ in Section~\ref{sect3} indeed was). We thus slightly modify our setting.

\begin{assum}\label{assdamp2}
We complement the Assumptions~\ref{assodampbasic} by the following.
\begin{enumerate}[(1)]
\item $V$ is a Banach space such that $V\hookrightarrow Y$.
\item $L$ can be extended to an operator that is bounded from $Y$ to $\partial X$, which we denote again by $L$, and such that $\ker(L)=V$.
\item $\begin{pmatrix} A\\ L\end{pmatrix}:D(A)\subset X\to X\times \partial X$ is closed.
\item ${A}_0:=A\vert_{D(A)\cap\ker(L)}$ has nonempty resolvent set.
\item $C$ is bounded from $[D(A)_L]$ to $X$.
\item $B_1,B_2$ are bounded from $[D(A)_L]$ to $\partial X$.
\item $B_3$ is bounded on $\partial X$.
\end{enumerate}
\end{assum}

\begin{rem}\label{com}
If $u$ is a classical solution to $({\rm AIBVP}^2_a)$, then $u\in C^1({\mathbb R}_+,Y)$. Since by Assumption~\ref{assdamp2}.(2) the operator $L$ is bounded from $Y$ to $\partial X$, we obtain that $L$ and the derivation with respect to time commute. In other words, if $w(t)=Lu(t)$ holds for all $t\geq 0$, then also $\dot{w}(t)=L\dot{u}(t)$ holds for all $t\geq 0$. 
\end{rem}

Thus, we are led to consider in this section a modified version of $({\rm AIBVP}^2_a)$, namely
\begin{equation}\tag{AIBPV$_b^2$}
\left\{
\begin{array}{rcll}
 \ddot{u}(t)&=& Au(t)+C\dot{u}(t), &t{\geq 0},\\
 \ddot{w}(t)&=& B_1u(t)+ B_2\dot{u}(t)+B_3 w(t)+B_4\dot{w}(t), &t{\geq 0},\\
 w(t)&=&Lu(t), \qquad \dot{w}(t)=L\dot{u}(t), &t{\geq 0},\\
 u(0)&=&f\in X, \qquad\;\;\dot{u}(0)=g\in X,&\\
 x(0)&=&h\in \partial X, \qquad\dot{x}(0)=j\in \partial X.&
\end{array}
\right.
\end{equation}
As before, we can perform a first order reduction of such a problem, re-writing it as $({\mathbb{ACP}})$.

We investigate $(\mathbb{ACP})$ on the non-diagonal Banach space $\underline{\mathbb{X}}$ defined by
$$\underline{\mathbb{X}}:=\left\{\begin{pmatrix}u\\ x\end{pmatrix}\in Y\times \partial X: Lu=x\right\}\times X\times\partial X\subset{\mathbb X},$$
instead of $\mathbb X$ as in Section~\ref{sect3}. A general function ${\mathbb u}:\mathbb{R}_+\to \underline{\mathbb{X}}$ is of the form
$$\mathbb{u}(t)\equiv\begin{pmatrix}u(t)\\ Lu(t)\\ v(t)\\ y(t)\end{pmatrix},\qquad t\geq 0.$$
Observe that if $\mathbb{u}(\cdot)$ is a classical solution to the problem $(\mathbb{ACP})$ in $\underline{\mathbb X}$, then by definition we obtain
$$\frac{du}{dt}(\cdot)=v(\cdot)\qquad \hbox{and}\qquad \frac{d(Lu)}{dt}(\cdot)=y(\cdot)$$
Again because $L$ and $\frac{d}{dt}$ commute for $u\in C^1({\mathbb R}_+,Y)$, we thus conclude that
\begin{equation*}
Lv(\cdot)=y(\cdot).
\end{equation*}

\begin{prop}\label{classol}
Consider the operator matrix $\underline{\mathbb A}$ on $\underline{\mathbb X}$ defined by 
\begin{equation}\label{dampabb3}
\underline{\mathbb A}=
\begin{pmatrix}
0 & 0 & I_Y & 0\\
0 & 0 & 0 & I_{\partial X}\\
A & 0 & C & 0\\
B_1 & B_3 & B_2 & B_4
\end{pmatrix}.
\end{equation}
with domain
\begin{equation}\label{dampabb3dom}
D(\underline{\mathbb A})=\left\{
\begin{pmatrix}u\\ x\\ v\\ y\end{pmatrix}
\in D(A)\times \partial X\times D(C)\times D(B_4): Lu=x,\; Lv=y
\right\}.
\end{equation}
Then $\mathbb{u}\in C^1({\mathbb R}_+,\underline{\mathbb X})$ is a classical solution to the initial-value problem associated with
$$ \dot{\mathbb{u}}(t)= {\mathbb A}\mathbb{u}(t), \qquad t{\geq 0},$$
if and only if it is a classical solution to the initial-value problem associated with
$$ \dot{\mathbb{u}}(t)= \underline{\mathbb A}\mathbb{u}(t), \qquad t{\geq 0}.$$
\end{prop}

\begin{cor}\label{classol2}
The well-posedness (in the classical sense) of the first order abstract Cauchy problem
\begin{equation}\tag{$\underline{\mathbb{ACP}}$}
\left\{
\begin{array}{rcl}
 \dot{\mathbb u}(t)&=&\underline{\mathbb A}{\mathbb u}(t),\qquad t\geq0,\\
 {\mathbb u}(0)&=&{\mathbb f},
\end{array}
\right.
\end{equation}
on $\mathbb X$ is equivalent to the well-posedness (in the classical sense) of the $({\rm AIBPV}_b^2)$ on $X$ and $ \partial X$.
\end{cor}

The following can be shown essentially like in the proof of~\cite[Lemma~5.2]{Mu05}. We denote by $D_\lambda^{A,L}$ the Dirichlet operators associated with $A,L$  introduced in Section~\ref{sect3}.

\begin{lemma}\label{uliso}
The product Banach space 
$$\underline{\mathbb V}:= V\times X\times \partial X\times \partial X$$
is isomorphic to $\underline{\mathbb X}$ via
\begin{equation*}
\underline{\mathbb U}_\lambda:=
\begin{pmatrix}
I_Y & -D_\lambda^{A,L} & 0 & 0\\
0 & 0 & I_X & -D_\lambda^{A,L}\\
0 & I_{\partial X} & 0 & 0\\
0 & 0 & 0 & I_{\partial X}\\
\end{pmatrix},\qquad \lambda\in\rho(A_0).
\end{equation*}
The inverse of $\underline{\mathbb U}_\lambda$ is the operator matrix
\begin{equation*}
\underline{\mathbb U}^{-1}_\lambda:=
\begin{pmatrix}
I_V &  & D_\lambda^{A,L} & 0\\
0 & 0 & I_{\partial X} & 0\\
0 & I_X & 0 & D_\lambda^{A,L}\\
0 & 0 & 0 & I_{\partial X}\\
\end{pmatrix},\qquad \lambda\in\rho(A_0).
\end{equation*}
\end{lemma}

In the remainder of this section we hence take $\lambda\in\rho(A_0)$ and investigate properties of the similar operator matrix $\underline{\mathbb U}_\lambda \underline{\mathbb A}\underline{\mathbb U}^{-1}_\lambda$ on the product space $\mathbb V$.

A tedious but direct matrix computation, similar to that performed in the proof of Lemma~\ref{simildamp}, yields the following.

\begin{lemma}\label{simildamp2}
Let $\lambda\in\rho(A_0)$. Then the operator matrix $\underline{\mathbb A}$ on $\underline{\mathbb X}$ defined in~\eqref{dampabb3}--\eqref{dampabb3dom} is similar to 
\begin{equation}\label{dampgbbunderline}
\underline{\mathbb G}:=
\begin{pmatrix}
\lineskip=0pt
0  & I_V & \linie & (*)\\ \back
A_0-D^{A,L}_\lambda B_1  & C_0-D^{A,L}_\lambda B_2 &\linie  & & \\
\noalign{\hrule}
0  & 0 &\linie  & 0  & I_{\partial X}\\ \back
B_1 & B_2 & \linie & B_3 + B_1 D_\lambda^{A,L} & B_4 + B_2 D_\lambda^{A,L}\\
\end{pmatrix}
\end{equation}
with domain
\begin{equation*}
D(\underline{\mathbb G}):= D(A_0)\times D(C_0)\times \partial X\times D(B_4)
\end{equation*}
on the Banach space $\underline{\mathbb V}$. Here the upper-right block entry $(*)$ is given by
\begin{equation*}
(*)=\begin{pmatrix}
0 & 0\\
-D_\lambda^{A,L}(B_1 D_\lambda^{A,L} +B_3)  & -D_\lambda^{A,L}(B_2 D_\lambda^{A,L}+B_4) + (C- \lambda)D_\lambda^{A,L}
\end{pmatrix}.
\end{equation*}

The similarity transformation is performed by means of the operator matrix $\underline{\mathbb U}_\lambda$ introduced in Lemma~\ref{uliso}.
\end{lemma}

Observe that if $B_4\in{\mathcal L}(\partial X)$, then the lower-right entry
\begin{equation}\label{redmat2}
\begin{pmatrix}
0  & I_{\partial X}\\
B_3 + B_1 D_\lambda^{A,L} & B_4 + B_2 D_\lambda^{A,L}
\end{pmatrix}
\end{equation}
in~\eqref{dampgbbunderline} is by assumption a bounded operator on $\partial X\times \partial X$. The following parallels	 parallels Theorem~\ref{maindampun}.

\begin{theo}\label{maindampbdd} 
Under the Assumptions~\ref{assodampbasic} and~\ref{assdamp2} the following assertions hold.
\begin{enumerate}[(1)]
\item Let 
$B_4\in{\mathcal L}(\partial X)$. Then the operator matrix $\underline{\mathbb A}$ generates a $C_0$-semigroup (resp., an analytic semigroup) on $\underline{\mathbb X}$ if and only if
\begin{equation}\label{redmat1}
\begin{pmatrix}
0 & I_V\\
{A}_0-D^{A,L}_\lambda B_1 & C_0-D^{A,L}_\lambda B_2
\end{pmatrix}
\quad\hbox{with domain}\quad D({A}_0)\times V
\end{equation}
generates a $C_0$-semigroup  (resp., an analytic semigroup) on $V\times X$ for some $\lambda\in\rho(A_0)$.

%
 \item Let $B_1=B_2=0$ and $D_\lambda^{A,L}B_4\in{\mathcal L}(\partial X)$ for some $\lambda\in\rho(A_0)$. If the semigroup generated by either of the
 matrices defined in~\eqref{redmat2}--\eqref{redmat1} is bounded and the other one is uniformly exponentially
 stable, then the semigroup generated by $\underline{\mathbb A}$ is bounded. 

\item Let $B_1=B_2=0$. Assume the semigroups generated by matrices in~\eqref{redmat2}--\eqref{redmat1} to be bounded. Let further the semigroup generated by the matrix in~\eqref{redmat2} be analytic. If the matrices  in~\eqref{redmat2}--\eqref{redmat1} have no common purely imaginary spectral values, then the semigroup generated by ${\mathbb A}$ is bounded.

\item Let $B_1\in{\mathcal L}(V,\partial X)$ and $B_2\in{\mathcal L}(X,\partial X)$. Assume both semigroups generated by matrices in~\eqref{redmat2}--\eqref{redmat1} to be uniformly exponentially stable. Then there is $\epsilon>0$ such that the semigroup generated
 by $\underline{\mathbb A}$ is uniformly exponentially stable whenever $\| B_1\|+\| B_2\|<\epsilon$.

\item The operator matrix $\underline{\mathbb A}$ on $\underline{\mathbb X}$ has compact resolvent if and only if both the embeddings $[D({A}_0)]\hookrightarrow V\hookrightarrow X$ are compact and ${\rm dim}\;\partial X<\infty$.
 \end{enumerate}
\end{theo} 

\begin{proof}
By Lemma~\ref{simildamp2}  $\underline{\mathbb A}$ is a generator on $\underline{\mathbb X}$ if and only if  $\underline{\mathbb G}$ is a generator on $\underline{\mathbb Y}$. 
Hence the operator matrix with diagonal domain $\underline{\mathbb G}$ can be studied by means of the results in~\cite[\S~3]{Na89}. Observe that the upper-right block entry $(*)$ of $\underline{\mathbb G}$ is a bounded operator from $\partial X\times\partial X$ to $V\times X$, by assumption.

(1) 
Since $C\in{\mathcal L}([D(A)_L],X)$, the upper-right block-entry of~\eqref{dampgbbunderline} is a bounded operator from $\partial X\times \partial X$ to $V\times X$. Also the  lower-left block-entry is bounded from $D(A_0)\times V$ to $\partial X \times \partial X$, and the claim follows by~\cite[Cor.~3.2]{Na89}. 

(3)--(4)--(5) The claims follow by Proposition~\ref{complete} below.

(6) The operator matrix $\underline{\mathbb G}$ has compact resolvent if and only if its domain is compactly embedded in $\underline{\mathbb X}$, i.e., if the embedding
$[D(\underline{\mathbb G})]=[D(A_0)]\times [D(C_0)]\times [\partial X]\times [D(B_4)]\hookrightarrow V\times X\times \partial X\times \partial X$ is compact.
\end{proof}

\begin{exa}\label{exas3}
We discuss the initial value problem associated with 
\begin{equation*}
\left\{
\begin{array}{rcll}
 \ddot{u}_j(t,x)&=& u''_j(t,x), &t\in {\mathbb R},\; x\in(0,1),\; j=1,\ldots,E,\\
u_j(t,\mv_i)&=&u_\ell(t,\mv_i)=:d^u_i(t), &t\in\mathbb R,\; j,\ell=1,\ldots,E,\; i=1\ldots,V,\\
 \ddot{u}(t,i)&=& \sum_{j=1} p_{ih}\phi_{hj} \frac{\partial u_j}{\partial\nu} (t,h) + \sum_{h=1}^V m_{ih} d^u_h(t)+ \sum_{h=1}^V n_{ih} d^{\dot{u}}_h(t), &t\in{\mathbb R },\; i=1,\ldots,V,\\
\end{array}
\right.
\end{equation*}
on a network $G$ with $E$ edges and $V$ vertices. Here $M=(m_{ih})$, $N=(n_{ih})$, and $P=(p_{ih})$ are $V\times V$ matrices.
(We refer to~\cite{KMS06,Mu07} for the graph-theoretical notation as well as for references to this kind of problems.)
Let
$$Y:=(W^{1,p}(0,1))^E\cap C(G),\quad X:=(L^p(0,1))^E,\quad\hbox{and}\quad \partial X:={\mathbb C}^V,$$
for any $1\leq p<\infty$. Furthermore, we set
$$Au:=u'',\qquad\hbox{for all}\; u\in D(A):=(W^{2,p}(0,1))^E\cap C(G),\qquad C:=0.$$
Thus, the damping effect only appears in the boundary conditions. Moreover we consider an operator 
$$(B_1u)_i:=\sum_{j=1} \phi_{ij} \frac{\partial u_j}{\partial\nu} (t,i)\qquad\hbox{for all}\; u\in D(B_1):=D(A)$$
of Kirchhoff-type, and
$$B_2:=0,\qquad B_3:=M,\qquad 
B_4:=N.$$
Let 
$$Lu:=d^u,\qquad D(L):=Y,$$
so that
$$V=(W^{1,p}_0(0,1))^E\qquad\hbox{and}\qquad D(A_0)=(W^{2,p}(0,1)\cap W^{1,p}_0(0,1))^E,$$
i.e., $A_0$ can be seen as a diagonal operator matrix consisting of second derivatives on $E$ unconnected intervals, each equipped with Dirichlet boundary conditions. Following the proof of~\cite[Prop.~7.1]{Mu05} one can show that $A_0-D^{A,L}_\lambda B_1$ generates a cosine operator function on $X$ for some $\lambda\in\rho(A_0)$, hence the operator matrix in~\eqref{redmat1} generates a $C_0$-group. Thus, $\underline{\mathbb A}$ generates a $C_0$-group on $\underline{\mathbb X}$. This system is a generalisation of that considered in~\cite[\S~2]{CENP05} for $E=1$, $V=2$, and $B_1=B_3=0$. 

Let $P=0$ and $M,N$ be negative definite. Then the matrix defined in~\eqref{redmat2} has negative spectrum, hence it generates a semigroup that is uniformly exponentially stable. Since moreover the group generated by the matrix in~\eqref{redmat1} is bounded, we conclude by Theorem~\ref{maindampbdd}.(2) that the solution to the problem is bounded and asymptotically almost periodic (for positive time). In particular, the problem admits a unique classical (backward as well as forward) solution.
\qed
\end{exa}

We may sometimes interpret our dynamic boundary conditions as {Wentzell}-type ones.

\begin{prop}\label{dw}
Let $\underline{\mathbb A}$ generate an analytic semigroup on $\mathbb X$. Then the solution $u$ to $({\rm AIBVP}^2_b)$ satisfies the abstract Wentzell-type boundary conditions
\begin{equation}\label{wbc}
L\left(Au(t)+C\dot{u}(t)\right)=B_1u(t)+B_2\dot{u}(t)+B_3Lu(t)+B_4L\dot{u}(t), \qquad t> 0.
\end{equation}
If further $C$ maps $D(A)$ into $Y$, then $u$ satisfies in fact
\begin{equation*}
LAu(t)+LC\dot{u}(t)=B_1u(t)+B_2\dot{u}(t)+B_3Lu(t)+B_4L\dot{u}(t), \qquad t> 0.
\end{equation*}
\end{prop}

\begin{proof}
By assumption $(\mathbb{ACP})$ is governed by an analytic semigroup, thus for all initial data $\mathbb f\in \underline{\mathbb X}$ the orbit ${\mathbb u}(\cdot):=e^{\cdot \mathbb A}{\mathbb f}$ is of class $C^\infty\big((0,\infty); [D({\mathbb A})]\big)$. In particular, taking the first coordinate $u$ of $\mathbb u$ and recalling that $u$ is by definition the solution to $({\rm AIBVP}^2_b)$, we deduce that 
\begin{equation}\label{inda}
Au(t)+C\dot{u}(t)=\ddot{u}(t)\in D(A)\qquad\hbox{for all }t>0.
\end{equation}
By Assumption~\ref{assodampbasic}.(3) we can apply the operator $L$ to $Au(t)+C\dot{u}(t)$, $t>0$. By Remark~\ref{com} $L$ commutes with the derivation with respect to time, so that 
\begin{equation}\label{cool}
\ddot{w}(t)=L\ddot{u}(t)=L\big(Au(t)+C\dot{u}(t)\big)\qquad\hbox{for all } t>0.
\end{equation}
Plugging~\eqref{cool} into~\eqref{dbc} we finally obtain~\eqref{wbc}.

Let now $C$ map $D(A)$ into $Y$. Since also $\dot{u}(t)\in D(A)$, $t>0$, we obtain that $C\dot{u}(t)\in Y$, $t>0$, and we conclude that
$$Au(t)=\ddot{u}(t)-C\dot{u}(t)\in Y,\qquad t>0.$$
Summing up, we can apply $L$ to each addend on the LHS of~\eqref{inda}.
\end{proof}

\begin{exa}\label{ex2}
We revisit the first system considered in Example~\ref{motiv}. The associated initial value problem system is governed by an analytic semigroup and Proposition~\ref{dw}.(4) applies. The solution $u$ satisfies $\frac{\partial^k u}{\partial t^k}(t,\cdot)\in D(A) = H^4(0,1)$ for all $t>0$ and $k\in\mathbb N$, and in particular $\dot{u}(t,\cdot),\ddot{u}(t,\cdot)\in H^4(0,1)$, $t>0$. If follows that $\dot{u}''(t,\cdot)\in H^2(0,1)$ and $u''''(t,\cdot)=\dot{u}''(t,\cdot)-\ddot{u}(t,\cdot)\in H^2(0,1)$ for all $t>0$.

Thus, for $t>0$ we can evaluate $u''''(t,\cdot)$ and $\dot{u}''(t,\cdot)$ at the endpoints of the interval $[0,1]$. We conclude that the solution to the initial value problem associated with~\eqref{motiv1} also satisfies
\begin{equation*}
\begin{array}{rl}
&u''''(t,j)-\dot{u}''(t,j)+(-1)^{j+1} u'''(t,j)+(-1)^{j} u'(t,j)\\
&\qquad\qquad +(-1)^j \dot{u}'(t,j)-u(t,j)-\dot{u}(t,j)=0,\qquad j=0,1,\; t>0,
\end{array}
\end{equation*}
a Wentzell-type boundary condition.
\qed\end{exa}

\section{The strongly damped case}\label{sect5}

In this section we discuss the problem in a strongly damped setting, i.e., we assume that $C$ is ``more unbounded" than $A$, and modify our assumptions accordingly. We treat both case  $L\not\in{\mathcal L}(Y,\partial X)$ and $L\in{\mathcal L}(Y,\partial X)$

More precisely, we consider a complete second order abstract initial-boundary value
problems with dynamic boundary conditions of the form 
\begin{equation}\tag{AIBPV$_c^2$}
\left\{
\begin{array}{rcll}
 \ddot{u}(t)&=& Au(t)+C\dot{u}(t), &t{\geq 0},\\
 \ddot{w}(t)&=& B_1u(t)+ B_2\dot{u}(t)+B_3 w(t)+B_4\dot{w}(t), &t{\geq 0},\\
 \dot{w}(t)&=&L\dot{u}(t), &t{\geq 0},\\
 u(0)&=&f\in X, \qquad\;\;\dot{u}(0)=g\in X,&\\
 x(0)&=&h\in \partial X, \qquad\dot{x}(0)=j\in \partial X.&
\end{array}
\right.
\end{equation}
Observe that the coupling relation expressed by the third equation is not the same of $({\rm AIBPV}_a^2)$ or $({\rm AIBPV}_b^2)$.

 \begin{assum}\label{assoverdamp}
We complement the Assumptions~\ref{assodampbasic} by the following.
\begin{enumerate}[(1)]
\item $\begin{pmatrix} C\\ L\end{pmatrix}:D(C)\subset X\to X\times \partial X$ is closed.
\item $C_0:=C_{|{\rm ker}(L)}$ has nonempty resolvent set.
\item $A$ is closed, $D(C)\subset D(A)$, and $[D(A)]$ is isomorphic to $Y$.
\item $\partial Y$ is a Banach space such that $[D(B_4)]\hookrightarrow \partial Y\hookrightarrow \partial X$.
\end{enumerate}
\end{assum}

As in Section~3, we denote by $[D(C)_L]$ the Banach space obtained by endowing $D(C)$ with the graph norm of the $C\choose L$, and for $\lambda\in\rho(C_0)$ we consider the Dirichlet operators $D^{C,L}_\lambda$ associated with $C,L$, which are bounded from $\partial X$ to $[D(C)_L]$. By the Closed Graph Theorem we further have that $[D(C)_L]\hookrightarrow Y$.

\begin{lemma}
Define the linear space
\begin{equation}\label{dampabb2dombis}
D_d({\mathbb A}):=\left\{
\begin{pmatrix}u\\ x\\ v\\ y\end{pmatrix}
\in D(A)\times D(B_3)\times D(C)\times D(B_4): Lv=y
\right\}.
\end{equation}
Consider the operator $\mathbb A$ with domain $D_d({\mathbb A})$ on the Banach space $\mathbb X$, where $\mathbb A$ and $\mathbb X$ are defined as in~\eqref{dampabb2} and~\eqref{xxx}.
Then the well-posedness of the first order abstract Cauchy problem $(\mathbb{ACP})$ on $\mathbb X$ is equivalent to the well-posedness of 
$({\rm AIBPV}_c^2)$ on $X$ and $ \partial X$.
\end{lemma}

With a proof similar to that of Lemma~\ref{simildamp}, one can see that the following holds.

\begin{lemma}\label{simildampbis}
Let $\lambda\in\rho (C_0)$. Then the operator matrix $({\mathbb A},D_d({\mathbb A})$ on $\mathbb X$ is similar to  
\begin{equation}\label{formh}
{\mathbb H}:=
\begin{pmatrix}
\lineskip=0pt
0  & I_Y & \linie & 0  & -D^{C,L}_\lambda\\ \back
A-D^{C,L}_\lambda B_1  & C_0-D^{C,L}_\lambda B_2 &\linie  & -D^{C,L}_\lambda B_3 & D^{C,L}_\lambda (\lambda- B_3 D^{C,L}_\lambda -B_4)\\ 
\noalign{\hrule}
0  & 0  & \linie & 0  & I_{\partial Y}\\ \back
B_1 & B_2 & \linie & B_3 & B_4+ B_2 D^{C,L}_\lambda\\
\end{pmatrix}
\end{equation}
with domain
\begin{equation*}
D({\mathbb H}):= D(A)\times D(C_0)\times D(B_3)\times D(B_4)
\end{equation*}
on the Banach space
$${\mathbb Y}:=Y\times X\times \partial Y\times \partial X.$$
The similarity transformation is performed by means of the operator matrix
\begin{equation*}
{\mathbb V}_\lambda:=
\begin{pmatrix}
I_Y & 0 & 0 & 0\\
0 & 0 & I_X & -D^{C,L}_\lambda\\
0 & I_{\partial Y} & 0 & 0\\
0 & 0 & 0 & I_{\partial X}\\
\end{pmatrix},
\end{equation*}
which is an isomorphism from $\mathbb X$ onto $\mathbb Y$, for any $\lambda\in\rho(C_0)$.
\end{lemma}


\begin{theo}\label{maindampunbis}
Under the Assumptions~\ref{assodampbasic} and~\ref{assoverdamp} the following assertions hold.
\begin{enumerate}[(1)]
\item Assume that  $B_1\in{\mathcal L}(Y,\partial X)$, and moreover that $B_2\in{\mathcal L}([D(C_0)],[D(B_4)])\cap{\mathcal L}([D(C)_L],\partial X)$ or else $B_2\in{\mathcal L}(X,\partial X)$. If $D^{C,L}_\lambda B_3\in{\mathcal L}(\partial Y,X)$ and $D^{C,L}_\lambda B_4\in{\mathcal L}(\partial X,X)$ for some $\lambda\in\rho(C_0)$, then  $\mathbb A$ generates a $C_0$-semigroup on $\mathbb X$ if and only if both $C_0-D^{C,L}_\lambda B_2$ and
\begin{equation}\label{redmatc}
\begin{pmatrix} 0& I_{\partial Y}\\
B_3 & B_4
\end{pmatrix}
\end{equation}
generate $C_0$-semigroups on $X$ and $\partial Y\times \partial X$, respectively.

\item Let $B_1\in{\mathcal L}(Y,\partial X)$ and $B_2\in{\mathcal L}(X,\partial X)$. If for some $\lambda\in\rho(C_0)$ both $C_0-D^{C,L}_\lambda B_2$ and the reduction matrix defined in~\eqref{redmatc} generate analytic semigroups on $X$ and $\partial Y\times \partial X$, respectively, then  $\mathbb A$ generates an analytic semigroup on $\mathbb X$.

\item Let $B_1\in{\mathcal L}(Y,\partial X)$ and $B_2\in{\mathcal L}([D(C)_L],\partial X)$. Assume that for some $\lambda\in\rho(C_0)$ $D^{C,L}_\lambda B_3\in{\mathcal L}(\partial Y,X)$ and $D^{C,L}_\lambda B_4\in{\mathcal L}(\partial X,X)$. If $C_0-D^{C,L}_\lambda B_2$ and the reduction matrix defined in~\eqref{redmatc} generate analytic semigroups on $X$ and $\partial Y\times \partial X$, respectively, then  $\mathbb A$ generates an analytic semigroup on $\mathbb X$.


\item Assume that $B_1,B_2\in{\mathcal L}(Y,\partial X)$ and $B_3\in{\mathcal L}(\partial Y,\partial X)$. Then $\mathbb A$ generates a cosine operator function on $\mathbb X$ if and only both $C_0$ and $B_4$ generate cosine operator functions with associated phase spaces $Y\times X$ and $\partial Y\times \partial X$, respectively.
\end{enumerate}
\end{theo}

\begin{proof}
Let $\lambda\in\rho(C_0)$. By Lemma~\ref{simildampbis} instead of $\mathbb A$ on $\mathbb X$ it suffices to investigate the similar operator matrix $\mathbb H$ on $\mathbb Y$. 
We consider $\mathbb H$ as a $2\times 2$ operator matrix with diagonal domain, i.e.,
$${\mathbb H}={\mathbb H}_0+{\mathbb H}_1:=
\begin{pmatrix}
{\bf A} & {\bf B}\\
{\bf C} & {{\bf D}}
\end{pmatrix}, 
\qquad D(\mathbb{H})=D({\bf A})\times D({{\bf D}}),$$
where the $2\times 2$ block-entries ${\bf A},{\bf B},{\bf C},{\bf D}$ are defined as in~\eqref{formh}.

(1) Under our assumptions we have $B_2 D^{C,L}_\lambda\in{\mathcal L}(\partial X)$, so that both $\mathbf A$ and $\mathbf D$ generate $C_0$-semigroups on $Y\times X$ and $\partial Y\times \partial X$, respectively. Now the off-diagonal block-entries of $\mathbb H$ define an additive perturbation which is bounded either on $[D({\mathbb H})]$ or on $\mathbb X$, and by~\cite[Cor.~3.2]{Na89} the claim follows.

(2)--(3) By assumption, both $\mathbf A$ and $\mathbf D$ generate analytic semigroups on $Y\times X$ and $\partial Y\times \partial X$, respectively. Then in (2) $\mathbf C$ is bounded from $Y\times X$ to $\partial Y\times \partial X$ while $\mathbf B$ is bounded from $[D({\mathbf D})]$ to $Y\times X$, and in (3)  $\mathbf B$ is bounded from $\partial Y\times \partial  X$ to $Y\times X$ while $\mathbf C$ is bounded from $[D({\mathbf A})]$ to $\partial Y\times \partial X$, and by~\cite[Cor.~3.3]{Na89} the claim follows.


(4) By assumption, \cite[Prop.~6.1]{Mu06b} applies and $\bf A$ and $\bf D$ generate cosine operator functions with associated phase spaces $(Y\times Y)\times (Y\times X)$ and $(\partial Y\times\partial Y)\times (\partial Y\times\partial X)$, respectively. Moreover, observe that ${\bf B}\in{\mathcal L}([D({\bf D})],Y\times Y)$ and ${\bf C}\in{\mathcal L}(Y\times Y,\partial Y\times \partial X)$. Thus, by~\cite[Prop.~3.2]{Mu06b} $\mathbb H$ generates a cosine operator function on $\mathbb X$.
\end{proof}

Recall that any generator of a cosine operator function also generates an analytic semigroup of angle $\frac{\pi}{2}$, cf.~\cite[Thm.~3.14.17]{ABHN01}.


Let us now modify our framework in order to deal with a setting where the boundary operator $L$ is bounded from $Y$ to $\partial X$.

\begin{assum}\label{assdamp2bis}
We complement the Assumptions~\ref{assodampbasic} by the following.
\begin{enumerate}[(1)]
\item $V$ is a Banach space such that $V\hookrightarrow Y$.
\item $L$ can be extended to an operator that is bounded from $Y$ to $\partial X$, which we denote again by $L$, and such that $\ker(L)=V$.
\item $\begin{pmatrix} C\\ L\end{pmatrix}:D(C)\subset X\to X\times \partial X$ is closed.
\item ${C}_0:=C\vert_{D(C)\cap\ker(L)}$ has nonempty resolvent set.
\item $A$ is bounded from $[D(C)_L]$ to $X$.
\item $B_1,B_2$ are bounded from $[D(C)_L]$ to $\partial X$.
\item $B_3$ is bounded on $\partial X$.
\end{enumerate}
\end{assum}

Under the Assumptions~\ref{assdamp2bis} Remark~\ref{com}, Proposition~\ref{classol}, and Corollary~\ref{classol2} still hold. Thus, we discuss the generator property of the same operator matrix $(\underline{\mathbb A},D(\underline{\mathbb A}))$ on $\underline{\mathbb X}$ introduced in Proposition~\ref{classol}. Moreover, also Lemma~\ref{uliso} remains valid up to replacing  $\underline{\mathbb U}_\lambda$ therein by
\begin{equation*}
\underline{\mathbb V}_\lambda:=
\begin{pmatrix}
I_Y & -D_\lambda^{C,L} & 0 & 0\\
0 & 0 & I_X & -D_\lambda^{C,L}\\
0 & I_{\partial X} & 0 & 0\\
0 & 0 & 0 & I_{\partial X}\\
\end{pmatrix},\qquad \lambda\in\rho(A_0).
\end{equation*}

\begin{lemma}\label{simildamp2bis}
The operator matrix $\underline{\mathbb A}$ on $\underline{\mathbb X}$ defined in~\eqref{dampabb3}--\eqref{dampabb3dom} is similar to 
\begin{equation*}
\underline{\mathbb H}:=
\begin{pmatrix}
\lineskip=0pt
0  & I_V & \linie & (*)\\ \back
A_0-D^{C,L}_\lambda B_1  & C_0-D^{C,L}_\lambda B_2 &\linie  & & \\
\noalign{\hrule}
0  & 0 &\linie  & 0  & I_{\partial X}\\ \back
B_1 & B_2 & \linie & B_3 + B_1 D_\lambda^{C,L} & B_4 + B_2 D_\lambda^{C,L}\\
\end{pmatrix}
\end{equation*}
with domain
\begin{equation*}
D(\underline{\mathbb G}):= D(A_0)\times D(C_0)\times D(B_3)\times D(B_4)
\end{equation*}
on the Banach space $\underline{\mathbb V}$. Here the upper-right block entry $(*)$ is given by
\begin{equation*}
(*)=\begin{pmatrix}
0 & 0\\
-D_\lambda^{C,L}(B_1 D_\lambda^{C,L} +B_3)  & -D_\lambda^{C,L}(B_2 D_\lambda^{C,L}+B_4) + (\lambda-A)D_\lambda^{C,L}
\end{pmatrix}.
\end{equation*}
\end{lemma}

\begin{theo}\label{maindampbddbis} 
Under the Assumptions~\ref{assodampbasic} and~\ref{assdamp2bis} the following assertions hold.
\begin{enumerate}[(1)]
\item Let 
$B_4\in{\mathcal L}(\partial X)$. Then the operator matrix $\underline{\mathbb A}$ generates a $C_0$-semigroup (resp., an analytic semigroup) on $\underline{\mathbb X}$ if and only if
\begin{equation}\label{redmatd}
\begin{pmatrix}
0 & I_V\\
{A}_0-D^{C,L}_\lambda B_1 & C_0-D^{C,L}_\lambda B_2
\end{pmatrix}
\quad\hbox{with domain}\quad V\times D({C}_0)
\end{equation}
generates a $C_0$-semigroup  (resp., an analytic semigroup) on $V\times X$ for some $\lambda\in\rho(C_0)$.


\item Let $A\in{\mathcal L}(V,X)$, $B_1,B_2\in{\mathcal L}(V,\partial X)$ 
and $B_4\in{\mathcal L}(\partial X)$. Then  $\underline{\mathbb H}$ generates a cosine operator function on $\underline{\mathbb X}$ if and only if $C_0$ generates a cosine operator function with associated phase space $V\times X$.
\end{enumerate}
\end{theo} 

\begin{proof}
The assertion in (1) can be proved in a way similar to Theorem~\ref{assdamp2}. To show that (2) holds, observe that the lower-right block-entry of $\underline{\mathbb H}$ is a bounded operator on $\partial X\times \partial X$ (hence the generator of a cosine operator function on $\partial X\times\partial X$), and that by~\cite[Prop.~6.1]{Mu06b} the upper-left block-entry of $\mathbb H$ generates a cosine operator function with associated phase space $(V\times V)\times (V\times X)$. Then, by assumption the lower-left block-entry of $\underline{\mathbb H}$ is bounded from $V\times V$ to $\partial X\times \partial X$ and the upper-right one is bounded from $\partial X\times \partial X$ to $X\times X$. Thus, by~\cite[Prop.~3.2]{Mu06b} also $\mathbb H$ generates a cosine operator function on $\mathbb X$.
\end{proof}

\begin{rem}
Stability criteria like those stated in Theorems~\ref{maindampun} and~\ref{maindampbdd} could be easily formulated also in the contexts of Theorems~\ref{maindampunbis} and~\ref{maindampbddbis}. However, little is currently known about the asymptotical behavior of strongly damped systems, thus such criteria could be hardly checked in concrete cases. 
\end{rem}


The following can be proved similarly to Proposition~\ref{dw}.

\begin{prop}\label{dw3}
Let Theorem~\ref{maindampbddbis} apply. Then the solution $u$ to $({\rm AIBVP}^2_b)$ satisfies abstract Wentzell-type boundary conditions
\begin{equation*}
L\left(Au(t)+C\dot{u}(t)\right)=B_1u(t)+B_2\dot{u}(t)+B_3Lu(t)+B_4L\dot{u}(t), \qquad t> 0.
\end{equation*}
If further $A$ maps $D(C)$ into $Y$, then $u$ satisfies in fact
\begin{equation*}
LAu(t)+LC\dot{u}(t)=B_1u(t)+B_2\dot{u}(t)+B_3Lu(t)+B_4L\dot{u}(t), \qquad t> 0.
\end{equation*}
\end{prop}

\begin{exa}
We consider the initial value problem associated with the system
\begin{equation*}
\left\{
\begin{array}{rcll}
 \ddot{u}(t,x)&=& \alpha u''(t,x) +\dot{u}''(t,x), &t\geq 0,\; x\in (0,1),\\
 u(t,0)&=&\dot{u}(t,0)=0, &t\geq 0,\\
 \ddot{u}(t,1)&=& -\beta_1 u'(t,1)-\beta_2\dot{u}'(t,1)+\beta_3 u(t,1)+\beta_4\dot{u}(t,1), &t\geq 0,\\
\end{array}
\right.
\end{equation*}
where $\alpha,\beta_1,\beta_2,\beta_3,\beta_4\in\mathbb C$. A similar problem has been discussed in~\cite[\S~4]{GV94} and~\cite[\S~4]{CENP05}, in a Hilbert space setting only.
Reformulate it as $({\mathbb{ACP}})$ by setting
$$X:=L^p(0,1),\qquad Y:=\left\{u\in W^{2,p}(0,1):u(0)=0\right\}, \qquad {\partial X}:={\mathbb C},$$
for $1\leq p<\infty$. We define the linear operators  
$$A:=\alpha\frac{d^2}{dx^2},\qquad C:=\frac{d^2}{dx^2},\qquad D({A}):=D(C):=Y,$$
$$Lu:=u(1)\qquad \qquad \hbox{for all } u\in D(L):=Y,$$
$$B_i u:=-\beta_i u'(1),\qquad \hbox{for all } u\in D(B_i):=Y,\; i=1,2,$$
$$B_i:= \beta_i,\qquad i=3,4.$$
Thus, $\ker(L)=W^{2,p}(0,1)\cap W^{1,p}_0(0,1)$ and $C_0:=C_{\arrowvert \ker(L)}$ is the second derivative with Dirichlet boundary conditions, the generator of a cosine operator function on $L^p(0,1)$. One sees moreover that all the Assumptions~\ref{assodampbasic} and~\ref{assdamp2bis} are satisfied.
Since $B_1,B_2\in{\mathcal L}(W^{2,p}(0,1),{\mathbb C})$, the above initial-boundary value problem is governed by an analytic semigroup. 
 This semigroup yields a solution that satisfies Wentzell-type boundary conditions for $t>0$.
\qed\end{exa}


\section{A technical result}\label{appe}

Consider the operator matrix
$${\mathcal H}:=\begin{pmatrix}
H & J\\
K & L\end{pmatrix},\qquad D({\mathcal H}):=D(H)\times D(L).$$
on a product Banach space $E\times F$, where $H: E\to E$ and $L: F\to F$ are closed operators. Our aim is to discuss the stability of the semigroup generated by $\mathcal H$, in the spirit of~\cite{Na89}.
In fact, if $J\in{\mathcal L}(F,E)$ and $K\in{\mathcal L}(E,F)$, then $\mathcal H$ generates a $C_0$-semigroup on $E\times F$ if and only if $H$ and $L$ generate $C_0$-semigroups on $E$ and $F$, respectively, and in this case 
$(e^{t\mathcal H})_{t\geq 0}$ is given by the Dyson--Phillips series
$$\sum_{k=0}^\infty S_k(t),\qquad t\geq 0,$$
where
$$S_0(t):=\begin{pmatrix}
e^{tH} & 0\\ 0 & e^{tL}\end{pmatrix},\qquad t\geq 0,$$ 
and
$$S_k(t):=\int_0^t S_0(t-s)\begin{pmatrix} 0 & J\\ K & 0\end{pmatrix}S_{k-1}(s) ds,\qquad t\geq 0,\; k=1,2,\ldots.$$
If we denote by $S_{k}^{(ij)}(t)$ the $(i,j)$-entry of the operator matrix $S_k(t)$, $t\geq 0$, $1\leq i,j\leq 2$, then a direct matrix computation shows that
\begin{eqnarray*}
S_k(t)\begin{pmatrix}x\\ y\end{pmatrix}& =& \int_0^t \begin{pmatrix} e^{(t-s)H} & 0\\ 0 & e^{(t-s)L}\end{pmatrix} 
\begin{pmatrix} JS_{k-1}^{(21)}(s)x+ JS_{k-1}^{(22)}(s)y\\ KS_{k-1}^{(11)}(s)x+ KS_{k-1}^{(12)}(s)y\end{pmatrix} ds\\
& =& \int_0^t \begin{pmatrix} e^{(t-s)H}JS_{k-1}^{(21)}(s)x+ e^{(t-s)H}JS_{k-1}^{(22)}(s)y\\
e^{(t-s)L}KS_{k-1}^{(11)}(s)x+ e^{(t-s)L}KS_{k-1}^{(12)}(s)y
 \end{pmatrix} ds,
\end{eqnarray*}
i.e., $(S_k(t))_{t\ge 0}$ can be expressed in terms of vector-valued convolution,
\begin{equation}\label{formulaconv}
S_k(t)= \begin{pmatrix}  e^{\cdot H}*JS_{k-1}^{(21)}(\cdot)& e^{\cdot H}*JS_{k-1}^{(22)}(\cdot)\\
e^{\cdot L}*KS_{k-1}^{(11)}(\cdot ) & e^{\cdot L}*KS_{k-1}^{(12)}(\cdot ) \end{pmatrix}(t),\qquad t\geq 0.
\end{equation}

By known results on vector-valued convolution we can now obtain the following. 

\begin{prop}\label{complete}
Let $M_1,M_2\geq 1$ and $\epsilon_1,\epsilon_2\leq 0$ be constants such that
\begin{equation}\label{stab}
\Vert e^{tH}\Vert \leq M_1e^{\epsilon_1 t}\qquad \hbox{and}\qquad \Vert e^{tL}\Vert \leq M_2e^{\epsilon_2 t},\qquad t\geq 0.\end{equation}
Then the following assertions hold.
\begin{enumerate}
\item Let $J=0$. If $\epsilon_1<0$ or $\epsilon_2<0$, then the semigroup $(e^{t\mathcal H})_{t\geq 0}$  is bounded.
\item  Let $J=0$. If $(e^{tL})_{t\geq 0}$ is analytic and $\sigma(H)\cap \sigma (L)\cap i{\mathbb R}=\emptyset$, then  $(e^{t\mathcal H})_{t\geq 0}$  is bounded.
\item Let both $\epsilon_1<0$ and $\epsilon_2<0$. Assume that 
$$M:=\frac{M_1 M_2\Vert J\Vert \Vert K\Vert}{ \epsilon_1 \epsilon_2}<1.$$
Then $(e^{t\mathcal H})_{t\geq 0}$ is uniformly exponentially stable.

\end{enumerate}
\end{prop}

\begin{proof}
To begin with, one can prove by complete induction over $n$ that 
$$
S_{2n}^{(12)}=S_{2n}^{(21)}=S_{2n+1}^{(11)}=S_{2n+1}^{(22)}=0,\qquad n\in\mathbb N.
$$
(1) If $\epsilon_1<0$ and $K=0$, then one can check that 
\begin{equation}\label{reson}
S^{(12)}_1=e^{\cdot H}*f:=e^{\cdot H}*J e^{\cdot L}
\end{equation}
and moreover $S^{21}_1=0$ as well as $S_k=0$ for all $k=2,3,\ldots$. We first consider the case $\epsilon_1<0$. Then, by the Datko--Pazy theorem $e^{\cdot H}x\in L^1({\mathbb R}_+,E)$ for all $x\in E$. Since $e^{\cdot L}y\in L^\infty({\mathbb R}_+,F)$ for all $y\in F$, by the vector-valued Young inequality (see~\cite[Prop.~1.3.5]{ABHN01}) we see that $S^{(12)}_1 y\in L^\infty({\mathbb R}_+,E)$ for all $y\in F$. The case $\epsilon_2<0$ can be treated similarly. Thus, we have shown that $e^{\cdot\mathcal H}=S_0(\cdot)+S_1(\cdot)\in L^\infty({\mathbb R}_+,{\mathcal L}(E\times F))$.

(2) The Laplace transform $\hat{f} (\lambda)$ of $f$ defined in~\eqref{reson} is given by $KR(\lambda,L)$ for all $\lambda$ with ${\rm Re}\lambda>0$. This yields that the half-line spectrum ${\rm sp}(f)$ of $f$, defined as in~\cite[\S~4.4]{ABHN01}, is given by $\{\eta\in{\mathbb R}:i\eta\in\sigma(H)\}$. Thus,~\cite[Thm.~5.6.5]{ABHN01} yields that $S^{(12)}_1 y\in L^\infty({\mathbb R}_+,E)$ for all $y\in F$, and again $e^{\cdot\mathcal H}\in L^\infty({\mathbb R}_+,{\mathcal L}(E\times F))$.

(3) We prove by complete induction over $n$ that the estimates
\begin{equation}\label{esti1}
\Vert S_{2n}^{(11)}(\cdot)x\Vert_{L^1({\mathbb R}_+,E)}\leq M^n \frac{M_1}{|\epsilon_1|} \Vert x\Vert,\qquad x\in E,
\end{equation}
\begin{equation}\label{esti2}
\Vert S_{2n}^{(22)}(\cdot)y\Vert_{L^1({\mathbb R}_+,F)}\leq M^n \frac{M_2}{|\epsilon_2|} \Vert y\Vert,\qquad y\in F,
\end{equation}
\begin{equation}\label{esti3}
\Vert S_{2n+1}^{(12)}(\cdot)y\Vert_{L^1({\mathbb R}_+,E)}\leq M^n \frac{M_1 M_2 \Vert J\Vert}{\epsilon_1 \epsilon_2}\Vert y\Vert,\qquad y\in F,
\end{equation}
\begin{equation}\label{esti4}
\Vert S_{2n+1}^{(21)}(\cdot)x\Vert_{L^1({\mathbb R}_+,F)}\leq M^n \frac{M_1 M_2 \Vert K\Vert}{\epsilon_1 \epsilon_2} \Vert x\Vert,\qquad x\in E,
\end{equation}
hold for all $k\in\mathbb N$. For $n=0$, one sees that~\eqref{esti1}--\eqref{esti2} are satisfied, since by~\eqref{stab}
$\Vert e^{\cdot H}x\Vert_{L^1}\leq -\frac{M_1}{\epsilon_1}\Vert x\Vert$ and $\Vert e^{\cdot L}y\Vert_{L^1}\leq -\frac{M_2}{\epsilon_2}\Vert x\Vert$. Moreover, $S_1^{(12)}(t)y=(e^{\cdot H}*J e^{\cdot L})(t)y$ by~\eqref{formulaconv}. By the Young inequality we also obtain
$$\Vert S_1^{(12)}(\cdot)y\Vert_{L^1}=\Vert e^{\cdot H}* Je^{\cdot L}y\Vert_{L^1}\leq\Vert e^{\cdot H}\Vert_{L^1} \Vert J\Vert \Vert e^{\cdot L}y\Vert_{L^1}\leq \frac{M_1 M_2 \Vert J\Vert}{\epsilon_1\epsilon_2}\Vert y\Vert.$$ 
Likewise one can show that $\Vert S_1^{(21)}(\cdot)x\Vert_{L^1}=\Vert e^{\cdot L}* Ke^{\cdot H}x\Vert_{L^1}\leq \frac{M_1 M_2 \Vert K\Vert}{\epsilon_1\epsilon_2}\Vert x\Vert$, thus the above inequalities hold for $n=0$. Assume now that they hold for $n$. Then for one applies the Young inequality and obtains
\begin{eqnarray*}
\Vert S_{2n+2}^{(11)}(\cdot)x\Vert_{L^1}&=&\Vert e^{\cdot H}* JS_{2n+1}^{(21)}(\cdot)x\Vert_{L^1}\leq 
\Vert e^{\cdot H}\Vert_{L^1}\Vert J\Vert \Vert S_{2n+1}^{(21)}(\cdot)x\Vert_{L^1}\\
&\leq& M^n\frac{M_1M_2\Vert J\Vert \Vert K\Vert}{\epsilon_1\epsilon_2} \frac{M_1}{|\epsilon_1|}\Vert x\Vert= M^{n+1}\frac{M_1}{|\epsilon_1|}\Vert x\Vert.
\end{eqnarray*}
The remaining three estimates can be proven likewise, leading to 
$$\Vert S_{2n}(\cdot)\begin{pmatrix}x\\y\end{pmatrix}\Vert_{L^1}+
\Vert S_{2n+1}(\cdot)\begin{pmatrix}x\\y\end{pmatrix}\Vert_{L^1}\leq M^n M_0\left\| \begin{pmatrix}x\\y\end{pmatrix}\right\|,\quad n\in{\mathbb N},\; \begin{pmatrix}x\\y\end{pmatrix}\in E\times F.$$
where $M_0:=\left(\frac{M_1}{|\epsilon_1|}+\frac{M_2}{|\epsilon_2|}+\frac{M_1M_2 \Vert J\Vert}{\epsilon_1\epsilon_2}+\frac{M_1M_2 \Vert K\Vert}{\epsilon_1\epsilon_2}\right)$.

We now prove the proposition's claim. 
By assumption $M<1$, thus
 by the dominated convergence theorem
\begin{eqnarray*}
\int_0^\infty\Vert e^{t\mathcal H}\begin{pmatrix}x\\ y\end{pmatrix}\Vert dt &\leq & \int_0^\infty \left(\sum_{k=0}^\infty \Vert S_k(t)\begin{pmatrix}x\\ y\end{pmatrix}\Vert\right) dt \leq  \sum_{k=0}^\infty \Vert S_k(\cdot)\begin{pmatrix}x\\y\end{pmatrix}\Vert_{L^1({\mathbb R}_+,E\times F)}\\
&\leq & M_0 \sum_{k=0}^\infty M^k\left\| \begin{pmatrix}x\\y\end{pmatrix}\right\|=\frac{M_0}{1-M}\left\| \begin{pmatrix}x\\y\end{pmatrix}\right\|,\qquad \begin{pmatrix}x\\y\end{pmatrix}\in E\times F.
\end{eqnarray*}
By the theorem of Datko--Pazy, this concludes the proof.
\end{proof}

\end{document}